\documentclass[a4paper,reqno]{amsart}

\usepackage{amssymb,amsmath}
\usepackage{amscd}

\font\tenmsb=msbm10 \font\sevenmsb=msbm7 \font\fivemsb=msbm5
\newfam\msbfam
\textfont\msbfam=\tenmsb \scriptfont\msbfam=\sevenmsb
\scriptscriptfont\msbfam=\fivemsb

\font\teneufm=eufm10 \font\seveneufm=eufm7 \font\fiveeufm=eufm5
\newfam\eufmfam
\textfont\eufmfam=\teneufm \scriptfont\eufmfam=\seveneufm
\scriptscriptfont\eufmfam=\fiveeufm

\def\co#1{
}

\renewcommand{\epsilon}{\varepsilon}
\renewcommand{\setminus}{\smallsetminus}

\def \normal{\triangleleft }

\newcommand{\Q}{\mathbb Q}
\newcommand{\Z}{\mathbb Z}

\newcommand{\fpinfty}{{\FP}_{\infty}}

\newcommand{\FP}{\operatorname{FP}}

\newcommand{\UUFP}{\underline{\underline{\operatorname{FP}}}}
\newcommand{\UFP}{\underline{\operatorname{FP}}}

\newcommand{\cohom}[3]{H^{{\raise1pt\hbox{$\scriptstyle#1$}}}(#2\>\!,#3)}
\newcommand{\tatecohom}[3]%
  {\widehat H^{{\raise1pt\hbox{$\scriptstyle#1$}}}(#2\>\!,#3)}

\newcommand{\Cohom}[3]%
  {H^{{\raise1pt\hbox{$\scriptstyle#1$}}}\big(#2\>\!,#3\big)}
\newcommand{\Tatecohom}[3]%
  {\widehat H^{{\raise1pt\hbox{$\scriptstyle#1$}}}\big(#2\>\!,#3\big)}

\newcommand{\homol}[3]{H_{{\lower1pt\hbox{$\scriptstyle#1$}}}(#2\>\!,#3)}
\newcommand{\homolog}[2]{H_{{\lower1pt\hbox{$\scriptstyle#1$}}}(#2)}

\newcommand{\mono}{\rightarrowtail}
\newcommand{\epi}{\twoheadrightarrow}

\newcommand{\eg}{{\underline EG}}

\newcommand{\uueg}{{\underline{\underline E}}G}

\newcommand{\blah}{{\phantom a}}

\newcommand{\OXG}{\mathcal O_{\mathcal X}G}
\newcommand{\OFG}{\mathcal O_{\mathcal F}G}
\newcommand{\OVCG}{\mathcal O_{\mathcal{VC}}G}

\newcommand{\uuz}{\underline{\underline\Z}}

\newtheorem{thm}{Theorem}[section]
\newtheorem{cor}[thm]{Corollary}
\newtheorem{prop}[thm]{Proposition}
\newtheorem{lemma}[thm]{Lemma}

\newtheorem{conjecture}[thm]{Conjecture}
\newenvironment{dem}{\noindent \underline{Proof:}}{\hfill$\square$\medskip}
\theoremstyle{remark}
\newtheorem{remark}[thm]{Remark}

\title{Cohomological finiteness conditions in Bredon cohomology}
\author{D.~ H. ~Kochloukova}
\address{Dessislava H.~Kochloukova, Department of Mathematics, State University of Campinas, Cx. P. 6065,
13083-970 Campinas, SP, Brazil}
\email{desi@ime.unicamp.br}
\author{C.~Mart\'inez-P\'erez}
\address{Conchita Mart\'inez-P\'erez, Departamento de Matem\'aticas, Universidad de Zaragoza,
50009 Zaragoza, Spain} \email{conmar@unizar.es}
\author{B.~ E.~A.~Nucinkis}
\address{Brita E.~A.~Nucinkis, School of Mathematics, University of Southampton, Southampton,
SO17 1BJ, United Kingdom}
\email{bean@soton.ac.uk}

\date{\today} 

\thanks{This work was supported by EPSRC grant EP/F045395/1  and  LMS Scheme 4 grant 4708.
The first named author was partially supported by "bolsa de produtividade em pesquisa", CNPq, Brazil. 
The second named author was also supported by Gobierno de Aragon and
MTM2007-68010-C03-01}

\begin{document}

\maketitle

\thispagestyle{empty}

\begin{abstract} We show that any soluble group $G$ of type Bredon-$\FP_{\infty}$ 
with respect to the family of all virtually cyclic subgroups such that centralizers of infinite order elements are of  type $\FP_{\infty}$ must be virtually cyclic. To prove this, we first reduce the problem to the case of polycyclic groups and then we show  that a polycyclic-by-finite group with finitely many conjugacy classes of maximal virtually cyclic subgroups is virtually cyclic. Finally we discuss refinements of this result:  we  only impose the property Bredon-$\FP_n$  for some $n \leq 3$ and restrict to  abelian-by-nilpotent, abelian-by-polycyclic or (nilpotent of class 2)-by-abelian groups.
\end{abstract}

\section{Introduction}


In this paper we study soluble groups $G$ of type Bredon-$\FP_{\infty}$  with respect to the family ${\mathcal{VC}}$ of all virtually cyclic subgroups of $G$.  Bredon cohomology with respect to a family $\mathcal X$ of subgroups plays a role in studying classifying spaces for families similar to the role played by ordinary cohomology in studying Eilenberg-Mac Lane spaces.  By a family $\mathcal X$ of subgroups we mean a set of subgroups of $G$ which is closed under conjugation by elements of $G$ and taking subgroups. A $G$-CW complex $X$ is said to be a model for a classifying space for a family $\mathcal X$ of subgroups of $G$, denoted by $E_{\mathcal X}G$,
if $X^H$ is contractible if $H \in {\mathcal X}$ and empty otherwise.  For the family $\mathcal F$ of finite subgroups, the classifying space for proper actions, also denoted by $\eg$, has received a lot of attention. The classifying space for the family ${\mathcal{VC}}$, denoted by $\uueg$, has only recently been looked at \cite{jpl, lueckweiermann}.
In \cite{jpl} the following  was conjectured:

\begin{conjecture}\label{conj}\cite [Juan-Pineda and Leary]{jpl}
Let $G$ be a group admitting a finite model for
 $\uueg$. Then $G$ is virtually cyclic.
\end{conjecture}

\noindent In \cite{lueckweiermann}  the dimension of the classifying
space $\uueg$ for virtually polycyclic groups is given in terms of
the Hirsch length. 

\noindent In this note we shall address the above conjecture purely algebraically by considering finiteness conditions in Bredon cohomology. This is motivated by the fact that a group $G$ has a finite type model for a classifying space with isotropy in a family $\mathcal X$ of subgroups of $G$ if and only if
$G$ is of type Bredon-$\fpinfty$ with respect to $\mathcal X$ and there is a model for a classifying space with finite $2$-skeleton \cite{lm}.

\noindent
Finiteness conditions for the family ${\mathcal F}$ of all finite subgroups in soluble groups are by now very well understood. Recently it was shown that a soluble group of type $\FP_{\infty}$ is of type Bredon-$\FP_{\infty}$ for the class $\mathcal F$ \cite{martineznucinkis} and admits a finite model for $\eg$ \cite{kmn}.  The proofs there utilised the following result which can be seen as the algebraic counterpart of a Theorem given by L\"uck \cite[Theorem 4.2]{lueck}:
 A group $G$ is of type Bredon-$\FP_\infty$ with respect to the family $\mathcal F$ if and only if $G$ has finitely many conjugacy classes of finite subgroups and each centraliser $C_G(K)$ of a finite subgroup $K$ is of type $\fpinfty.$ For a proof of this fact see \cite{kmn}.

\noindent The main result, Theorem \ref{mainresult} gives an answer to Conjecture \ref{conj} for certain soluble
groups.  As we will also show, this result can easily be extended to elementary amenable groups.

\medskip\noindent{\bf Theorem 5.1.} {\sl Soluble groups $G$ of type Bredon-$\FP_{\infty}$ with respect to the class of virtually cyclic groups such that  centralizers of infinite order elements are of type $FP_\infty$ are virtually cyclic.}

\noindent Part of the proof of Theorem \ref{mainresult} reduces 
to the study of groups where centralizers of cyclic subgroups are
finitely generated. This is not a new problem and some partial
answers can be found in the work of J. Lennox \cite{Lennox2}. In
\cite{Lennox1} J. Lennox asked whether a soluble group $G$ is
polycyclic if all centralizers of its finitely generated subgroups
are finitely generated. A positive answer to this question was given
in \cite{Lennox2} when $G$ is nilpotent-by-polycyclic and this
result was strengthened further by showing  that a group $G$ with an
abelian normal subgroup $A$ such that $G/ A$ is nilpotent and the
centralizer (in $G$) of every element of $A$ is finitely generated,
is polycyclic. In the same paper it was also conjectured that the
same result holds for abelian-by-polycyclic groups.  In Theorem
\ref{FP3} we give an affirmative answer of this problem for
abelian-by-polycyclic groups $G$ of homological type $\FP_3$ using
the recent result that such  a group $G$ is
nilpotent-by-abelian-by-finite \cite{GKR}.  It is still an open
question whether a soluble group $G$ is polycyclic if the
centralizer of any element of $G$ is finitely generated.

The final step of the proof of Theorem \ref{mainresult} involves conjugacy classes of maximal virtually cyclic subgroups. More precisely we show that if a polycyclic-by-finite group has finitely many conjugacy classes of maximal virtually cyclic subgroups then it is virtually cyclic. We show a more general result:  Every nilpotent-by-abelian-by-finite group with the maximal condition on virtually cyclic subgroups and finitely many conjugacy classes of maximal virtually cyclic subgroups is virtually cyclic, see Theorem \ref{quotientnaf}.

\bigskip\noindent{\bf Remark added in 2018.} Since the publication of the original version of this article in 2011, Conjecture \ref{conj} has been proven for  elementary amenable groups \cite{gw}, one-relator groups, CAT(0)-groups, acylindrically hyperbolic groups, $3$-manifold groups \cite{vPW1}, and linear groups \cite{vPW2}. In all of these examples it was shown that these groups cannot be of type $\UUFP_0.$

\section{Bredon cohomology with respect to virtually cyclic subgroups}

\noindent Let us begin this section with recalling a few prerequisites from Bredon cohomology. Here we replace the group $G$ by the orbit category $\OXG$, where, as before, $\mathcal X$ denotes a family of subgroups of $G$.
The category $\OXG$ has as objects the transitive $G$-sets with stabilizers in $\mathcal X$. Morphisms in this category are $G$-maps.
 In this note we shall mainly be concerned with the families
$\mathcal F$ of finite subgroups  and $\mathcal{VC}$ of virtually cyclic subgroups and we  write $\OFG$ and $\OVCG$ for the orbit categories.  Modules over the orbit category, or $\OXG$-modules, are contravariant functors from
the orbit category to the category of abelian groups. Exactness is defined pointwise: A sequence $A\to B\to C$ of $\OXG$-modules
is exact at $B$ if and only if
$A(\Delta)\to B(\Delta)\to C(\Delta)$ is exact at $B(\Delta)$ for every transitive $G$-set $\Delta$.
The category of $\OXG$-modules has enough projectives, which are constructed as follows:
For any $G$-sets $\Delta$ and $\Omega$ we denote by $[\Delta,\Omega]$ the set of $G$-maps from $\Delta$ to $\Omega$.
Let $\Z[\Delta,\Omega]$ be the free abelian group on $[\Delta,\Omega]$. Now fix  $\Omega$ and let $\Delta$  range over the transitive
 $G$-sets with  stabilizers in $\mathcal X$. This gives rise to  an $\OXG$-module $\Z[\blah,\Omega]$.
Let $G/K$ be a transitive $G$-set with $K \in \mathcal X.$  Then $\Z[\blah, G/K]$ is a free module. Projective modules are defined analogously to the ordinary case.
The trivial $\OXG$-module, denoted $\Z(\blah)$, is the constant functor from $\OXG$ to the category of abelian groups. Bredon cohomology is now defined via projective resolutions of $\Z(\blah)$.
The notions of type Bredon-$\FP$, Bredon-$\FP_n$ ($n \geq 0$ an integer) and Bredon-$\FP_\infty$
are defined in terms of projective resolutions of $\Z(\blah)$ over $\OXG$ analogously to
the classical notions of type $\FP$, $\FP_n$ and $\fpinfty$.

To stay in line with notation previously used, we say a module is of type $\UFP_\infty$  if it is of type Bredon-$\fpinfty$ with respect to the family of all finite subgroups and is of type $\UUFP_\infty$ if it is of type Bredon-$\fpinfty$ with respect to the family of all virtually cyclic subgroups. We also denote by $\underline\Z$ and $\underline{\underline\Z}$ the respective constant modules.

\medskip
\noindent We shall now give a description of groups of type $\UUFP_n$ and $\UUFP_\infty$ in terms of cohomological finiteness conditions of centralisers of virtually cyclic subgroups and conjugacy in $\mathcal{VC}$.  Let us consider the following two conditions on a group $G$:
\begin{itemize}
\item[(1)] $G$ has the maximal condition on virtually cyclic subgroups (max-vc) and has finitely many conjugacy classes of maximal virtually cyclic subgroups.
\item[(2)] There are finitely many virtually cyclic groups $H_1,\ldots, H_n$ of $G$ such that every virtually cyclic  subgroup is subconjugate to one  of the $H_i$s.
\end{itemize}

\begin{remark} Obviously, (1) implies (2)  but the converse is not true in general. The following example was pointed out to us by Peter Kropholler. For a torison-free group $H$ the famous construction by Higman-Neumann-Neumann \cite[6.4.6]{robinson} embeds $H$ in a torsion-free group $H^*$ such that all non-trivial elements in $H$ are conjugated. In a torsion-free group every virtually cyclic group is cyclic so obviously $G = H^*$  satisfies (2).  But for certain particular choice of $H$, for example $H = \mathbb{Z} [{1 \over 2}]$, the group $G = H^*$ does not satisfy  (1).

\end{remark}

\noindent We have, however:

\begin{lemma}\label{max} If $G$ has property (2) and max-vc then it has property (1).
\end{lemma}
\begin{proof} We only have to prove that there are finitely many conjugacy classes of maximal
virtually cyclic subgroups. But this is obvious as the subgroups
$H_1,\ldots,H_n$ of condition (2) can in this case be taken to be
maximal virtually cyclic. And for any other maximal virtually cyclic
$T$, we have $T^x\leq H_i$ for some $x\in T$, some $i$. Then $T^x$
is also maximal virtually cyclic and therefore $T^x=H_i.$
\end{proof}

\noindent The following Lemma is an analogue to \cite[Lemma 3.1]{kmn}

\begin{lemma}\label{uufp0}
Let $G$ be a group. Then $G$ is of type  $\UUFP_0$ if and only if $G$ satisfies condition (2).
\end{lemma}

\begin{proof}
Let $G$ be of type $\UUFP_0$. Then the trivial module $\uuz$ has a finitely generated free Bredon module mapping onto it. Hence there is a $G$-finite $G$-set $\Omega$ with virtually cyclic stabilizers and an
epimorphism
$$\Z[\blah,\Omega]\epi\uuz.$$
Now let $K$ be an arbitrary virtually cyclic subgroup of $G$. Evaluating the above epimorphism at $G/K$ we obtain an
epimorphism $\Z\Omega^K\epi\Z$ and therefore $\Omega^K$ is non-empty. Hence $K$ is subconjugated to a representative of the finitely many conjugacy classes of subgroups which have fixed points in $\Omega$. Conversely, if $G$ satisfies condition (2), then we can take
$\displaystyle\Omega=\bigsqcup_{H_i}G/H_i$ where $H_i$ runs through a set of conjugacy class representatives
as above, and the obvious augmentation map $\Z[\blah,\Omega]\epi\uuz$ is an epimorphism.
\end{proof}

\noindent We say a Bredon module $M$ is finitely generated if there is a finite $\OXG$-set $\Sigma$ in the sense of L\"uck \cite[9.16, 9.19]{lueckbook} such that there is a free Bredon module $F$ on $\Sigma$ mapping onto $M$.  In particular, a projective Bredon-module $P$ is finitely generated if there is a $G$-finite $G$-set $\Omega$ with stabilisers in $\mathcal X$ such that $P$ is a direct summand of $\Z[\blah,\Omega].$


\
\begin{lemma}\label{fgproj}
Let $P(\blah )$ be a  finitely generated projective $\OVCG$-module, and let $K$ be a finite subgroup. Then, after evaluation, $P(G/K)$ is a module of type $\fpinfty$ over the Weyl-group $WK=N_G(K)/K.$
\end{lemma}

\begin{proof}  By the remarks above, it suffices to show that for each virtually cyclic group $H$, the module $\Z(G/H)^K \cong \Z[G/K,G/H]$ is of type $\FP_\infty$ as a $WK$-module. Furthermore,
$(G/H)^K$ is a $WK$-set in which each $xH$ is stabilized by $(N_G(K)\cap H^{x^{-1}})K /K = WK_x.$
Now we claim that 
$$\Z(G/H)^K = \bigoplus_{\begin{array}{c} { \small{x\in N_G(K)\backslash G/H}} \\ {\small{K^x\leq H}}\end{array}} \Z(WK/WK_x)$$
is a finitely generated $WK$-module. To see it, observe that there are finitely many $H$-conjugacy classes of finite subgroups of $H$ (see for example \cite[2.4]{martineznucinkis}) and this yields finitely many possible choices of the double class $N_G(K)xK$. 
An argument analogous to
\cite[Proposition 6.3]{nucinkis99} reduces to showing that each
$\Z(WK/WK_x)$ is of type $\FP_\infty$. But this follows from the
fact that $WK_x$ is virtually cyclic so it is a group of type
$\FP_\infty.$
\end{proof}

\begin{cor}
Let $M(\blah)$ be a a Bredon-module of type $\UUFP_n$. Then, for each finite subgroup $K$, the WK-module $M(G/K)$ is of type $\FP_n.$
\end{cor}

\begin{proof}
This is a straight forward dimension shift applying \cite[Proposition 1.4]{Bieribook}.
\end{proof}

\begin{prop}\label{uufpn}
Let $G$ be a group of type $\UUFP_n$. Then the following two equivalent statements hold:
\begin{enumerate}
\item $G$ satisfies condition (2) and for each finite subgroup $K$, the Weyl-group  $WK$ is of type $\FP_n.$
\item $G$ satisfies condition (2) and for each finite subgroup $K$, the centralizer $C_G(K)$ is of type $\FP_n.$
\end{enumerate}
\end{prop}

\begin{proof}  Statement (i) follows directly from Lemmas \ref{uufp0} and \ref{fgproj}. Since $K$ is finite, it follows that $N_G(K)$ is of type $\FP_n$ if and only $WK=N_G(K)/K$ is \cite[Proposition 2.7]{Bieribook}. $Aut(K)$ is a finite group and hence the index $|N_G(K):C_G(K)|$ is finite and therefore $N_G(K)$ is of type $\FP_n$ if and only $C_G(K)$ is of type $\FP_n$.
\end{proof}

\noindent Let us also remark on an obvious link with $\UFP_\infty.$

\begin{cor}
Let $G$ be a group of type $\UUFP_\infty$. Then $G$ is of type $\UFP_\infty.$
\end{cor}
\begin{proof} Note that by Proposition \ref{uufpn} the group satisfies (2) and this together with the fact that virtually cyclic groups have finitely many conjugacy classes of finite subgroups  implies that also $G$ has only a finite number of conjugacy classes of finite subgroups. Moreover, Proposition \ref{uufpn} also implies that for any finite $K\leq G$, $N_G(K)/K$ is $\FP_\infty$ and the result follows by \cite[3.1]{kmn}.
\end{proof}

\begin{prop}\cite[Corollary 5.4]{lueckweiermann}
Let $G$ be a group admitting a finite (finite type) model for $\uueg$. Then $G$ admits a finite (finite type) model for $\eg$.
\end{prop}

\section{Soluble groups}

\begin{lemma} \label{commutativealgebra} Let $Q$ be a finitely generated abelian group and $V$ be a finitely generated $\Z Q$-module. Then $V$ is finitely generated as an additive group if and only if $\Z Q / P_i$ is finitely generated as an additive group for every minimal associated prime $P_i$ of $V$.
\end{lemma}

\begin{proof} Suppose that $\Z Q / P_i$ is finitely generated for $1 \leq i \leq s$, where $P_1, \ldots, P_s$ are the minimal associated primes of $V$.
Let $I$ be the annihilator $ann_{\Z Q}(V)$. Then $\sqrt{I} = P_1
\cap \ldots \cap P_s$ and $\Z Q / P_1 \cap \ldots \cap P_s$ embeds
in $\Z Q / P_1 \oplus \ldots \oplus \Z Q/ P_s$ via the diagonal map
, so $\Z Q / \sqrt{I}$ is finitely generated as an additive group.
Let $s$ be a natural number such that
  $\sqrt{I}^s \subseteq I$. Note that every quotient $\sqrt{I}^j / \sqrt{I}^{j+1}$ is a finitely generated $\Z Q / \sqrt{I}$-module, so $\Z Q / \sqrt{I}^s$ and consequently $\Z Q / I$ are finitely generated as additive groups. Since $V$ is finitely generated as a $\Z Q/ I$-module, $V$ is finitely generated as an additive group.

Conversely suppose that $V$ is finitely generated as an additive group. Since $\Z Q/ P_1 \oplus \ldots \oplus \Z Q/ P_s$
embeds in $V$ it follows that every $\Z Q / P_i$ is finitely generated.
\end{proof}

  \begin{lemma} \label{FP2} Let $N_2 \mono G_2 \epi Q_2$ be a short exact sequence of groups with $G_2$ of type $\FP_2$. Suppose that $G_2$ and $N_1$ are subgroups of a group $G$, $G_2 \subseteq N_{G}(N_1)$, $N_1\cap G_2=N_2$ and $N_2$ is a normal subgroup  in $N_1$ such that $N_1 / N_2$ is finitely generated and abelian. Then $G_1 = N_1 G_2$ is of type $\FP_2$.
  \end{lemma}

  \begin{proof}  The lemma is obvious if $G_2$ is normal in $G_1$ but we cannot assume this.
  Since $N_1 / N_2$ is a finitely generated abelian group, it is a direct sum of cyclic groups. Let $T$ be a finite subset of $N_1$ such that the images of the elements of $T$ in $N_1 / N_2$ are the generators of the direct cyclic summands. We split $T$ as a disjoint union $T = T_0 \cup T_1$, where the images in $N_1 / N_2$ of the elements of $T_0$ have finite order
 and the   images in $N_1 / N_2$ of the elements of $T_1$ have infinite order.
  Write $T_0 = \{ t_1, \ldots , t_m \}$, $a_i$ is the order of the image of $t_i$ in $N_1 / N_2$, $T_1 = \{ t_{m+1} ,\ldots , t_s \}$.
  Fix a generating set $Y = \{g_1, \ldots, g_s \}$ of $G_2$.
Note that every element $g$ of $G_1$ can be written in a unique way as
  \begin{equation} \label{normalform}
  g = t_1^{z_1} \ldots t_s^{z_s} h \hbox{ where } h \in G_2, z_i \in \Z \hbox{ and } 0 \leq z_i < a_i \hbox{ for } i \leq m.
  \end{equation}
  Then for some reduced words $\alpha_{i,j}, \beta_i$ on $Y \cup Y^{-1}$ we have the following equations in $G_1$
 $$t_i t_j = t_j t_i \alpha_{i,j}  \hbox{ for } 1 \leq j < i \leq s ,
  t_i^{a_i} = \beta_i \hbox{ for } 1 \leq i \leq m$$
  and for $y \in Y \cup Y^{-1}$ there are reduced words $\gamma_{i,y}$ on $Y \cup Y^{-1}$ such that
  $$
  y t_i y^{-1} = t_1^{z_{1,i,y}} \ldots t_s^{z_{s,i,y}} \gamma_{i,y} \hbox{ where }
  z_{j,i,y} \in \Z \hbox{ and } 0 \leq z_{j,i,y} < a_j \hbox{ for } j \leq m
  $$

  Let $\langle Y \mid R \rangle$ be a presentation of $G_2$ that shows it is of type $\FP_2$ i.e. $Y$ is as above, $R$ might be infinite but the coresponding relation module is finitely generated as a $\Z G_2$-module. Then $G_1$ has a presentation $$\langle Y \cup T \mid R \cup S \rangle $$ where $S$ is $$ \{
  t_j^{-1} t_i^{-1} t_j t_i \alpha_{i,j} \}_{1 \leq j < i \leq s} \cup  \{ t_i^{a_i} \beta_i^{-1} \}_{m+1 \leq i \leq s } \cup
  \{ y t_i^{-1}  y^{-1} t_1^{z_{1,i,y}} \ldots t_s^{z_{s,i,y}} \gamma_{i,y}
    \}_{y \in Y \cup Y^{-1}, 1 \leq i \leq s}. $$ Indeed the group $H$ defined by the presentation $\langle Y \cup T \mid R \cup S \rangle $ maps surjectively to $G_1$  and the elements of $H$ satisfy the normal form condition (\ref{normalform}). Finally the presentation $\langle Y \cup T \mid R \cup S \rangle $ is obtained from the presentation $\langle Y \mid R \rangle$ of $G_2$ by adding finitely many relations and generators i.e. $G_1$ is of type $\FP_2$.
  \end{proof}

 \begin{thm} \label{polycyclic3}
 Let $N \mono G \epi Q$ be a short exact sequence of groups with $N$ nilpotent and $Q$ abelian, $G$ finitely generated such that every quotient of $G / Z(N)$ is of type $\FP_{2}$ and for every $g \in N$ the centraliser
 $C_G(g)$ is finitely generated and
  the image of  the centralizer $C_G(g)$ in any quotient of $G / Z(N)$ is of type $\FP_2$. Then $G$ is polycyclic.
 \end{thm}

 \begin{proof}

We show first that $Z(N)$ is finitely generated as a group.
 By assumption $G/Z(N)$ is of type $\FP_2$, thus the centre $Z(N)$ of $N$ is finitely generated as a $\Z Q$-module, where $Q$ acts via conjugation. Let $P_1, \ldots, P_s$ be the minimal associated primes for the $\Z Q$-module $Z(N)$, thus $M = \Z Q/ P_1 \oplus \ldots \oplus \Z Q / P_s $ embeds in $Z(N)$ as a $\Z Q$-submodule (i.e. if $v_i$ is an element of $Z(N)$ with $Ann_{\Z Q}(v_i) = P_i$
 then $M$ is isomorphic to $\sum_{1 \leq i \leq s} \Z Q v_i$)
 and we think of $M$ as a submodule of $Z(N)$. Let $a$ be the element of $M$ whose projection to $\Z Q / P_i$ is the unity element for every $1 \leq i \leq s$.

  By assumption  the centralizer $C_G(a)$ is finitely generated. Note that $N \subseteq C_G(a)$  and define $Q_0 = C_G(a) / N \subseteq Q$. Since $C_G(a) / Z(N)$ is of type $\FP_2$, $Z(N)$ is finitely generated as a $\Z Q_0$-module. Since $\Z Q_0$ is Noetherian, every $\Z Q_0$-submodule of $Z(N)$ is finitely generated. In particular $M$ is finitely generated as a $\Z Q_0$-module.
  Since $Q_0$ acts trivially on $a$ we deduce that $Q_0$ acts trivially on $M$. Thus $M$ is finitely generated as an additive abelian group and by Lemma \ref{commutativealgebra} $Z(N)$ is finitely generated.

We prove the theorem by  induction on the nilpotency class of $N$.
We aim to prove that the assumptions of the theorem hold for the quotient group $\overline{G} = G / Z(N)$ and then the proof can be completed by induction.
Write $\overline{g}$ for the image of $g \in G$ in $\overline{G}$ and define $\overline{N} = N / Z(N)$. For some fixed $g \in N$ define the homomorphsim
$$
\varphi : C_{\overline{N}}(\overline{g}) \to Z(N)
$$
that sends $\overline{n}$ to $[n,g]$. Note that $Ker(\varphi) = C_N(g) / Z(N)$ and that $C_{\overline{N}}(\overline{g}) / Ker(\varphi)$ is a subgroup of $Z(N)$, hence is a finitely generated abelian group. Then $G_0 = C_{\overline{G}}(\overline{g})$ has a filtration of subgroups
$$
G_2 = C_G(g) / Z (N) \subseteq G_1 = G_2 C_{\overline{N}}(\overline{g}) \subseteq G_0.
$$
Let $N_2 = C_N(g) / Z(N)$ and $N_1 = C_{\overline{N}}(\bar{g})$, so $N_1 / N_2$ is finitely generated and abelian. Note also that $N_1\cap G_2=N_2$. Then by Lemma \ref{FP2} $G_1$ is of type $\FP_2$. Since $G_1$ is normal in $G_0$ and $G_0 / G_1$ is finitely generated abelian, $G_0$ is of type $\FP_2$.

By assumption the image of $C_G(g)$ in any quotient of $G / Z(N)$ is of type $\FP_2$. Then the above argument applied for the images of the groups $G_2, G_1, G_0, N_1, N_2$ in a quotient of $G / Z(N)$ shows that the image of $G_0$ in any quotient of $G / Z(N)$ is of type $\FP_2$. Thus we can apply the inductive argument for $G / Z(N)$ and deduce that $G / Z(N)$ is polycyclic.

\end{proof}

\begin{cor} \label{class2}
 Let $N \mono G \epi Q$ be a short exact sequence of groups such that $N$ is nilpotent of class 2, $Q$ is abelian, $G$ is of type $\FP_2$ and for every $g \in N$ the centralizer $C_G(g)$ is of type $\FP_2$. Then $G$ is polycyclic.
 \end{cor}

\begin{proof} Every metabelian quotient of a group of type $\FP_2$  not containing non-cyclic free subgroups is of type $\FP_2$ \cite{BieriStrebel}. Then the previous theorem applies.
\end{proof}

There are known examples of groups $G$ that are split extensions of a nilpotent of class 3 group $N$ by a finite rank free abelian group $Q$ such that $G$ is finitely presented, $Z(N) = Z(G)$ is infinitely generated as a group and thus $G / Z(N)$ is not of type $\FP_2$ \cite{Abels}.
Though these examples show that the first condition of Theorem \ref{polycyclic3} is quite restrictive, at the same time these examples  do not satisfy the second type conditions about the centralizer, i.e. there is an element $g \in N$ such that
 $C_G(g)$ is not even finitely generated.

  \begin{lemma} \label{quotient} Let $G$ be a soluble group of type $\FP_{\infty}$. Then every quotient of $G$ is of type $\FP_{\infty}$.
\end{lemma}

\begin{proof} This follows directly by \cite[Thm.~4]{baumslagbieri} and the classification of soluble $\FP_{\infty}$ groups as constructible \cite{kropholler86}, \cite{kropholler93b}.

\end{proof}

\begin{cor} \label{centraliserFP}Let $G$ be a soluble group of type $\FP_{\infty}$ such that the centraliser $C_G(g)$ is of type $\FP_{\infty}$ for every element $g \in G$. Then $G$ is polycyclic.
\end{cor}

\begin{proof} By \cite[Corollary of Theorem B]{kropholler93b}, any soluble group of type $\FP_{\infty}$ is constructible and virtually of type $\FP$. This implies that the group has finite Pr\"ufer rank and therefore it is virtually abelian-by-nilpotent (see \cite[10.38]{robinson}).  Then the result follows directly from Theorem \ref{polycyclic3} and Lemma \ref{quotient}.
\end{proof}

In the following theorem a stronger condition on $G$ is imposed at the expense of relaxing the conditions on the centralizers. Note that we have not included that every quotient of $G$ is of type $\FP_3$ but this holds by one of the main results of \cite{GKR}.

\begin{thm} \label{FP3} Let $G$ be an abelian-by-polycyclic group of type $\FP_3$ such that the centraliser $C_G(g)$ is finitely generated for every element $g \in A$, where $A$ is a normal abelian  subgroup of $G$ with $G/A$ polycyclic. Then $G$ is polycyclic.
\end{thm}

\begin{proof} We aim to prove that $A$ is finitely generated as a group. We view $A$ as a $\Z H$-module with $H$-action induced by conjugation, where $H = G / A$. By \cite{GKR} $G$ is nilpotent-by-abelian-by-finite. By going down to a subgroup of finite index we can assume that $G$ is nilpotent-by-abelian, so the commutator $H'$ acts nilpotently on $A$ i.e. there is some natural number $s$ such that $A \Omega^s = 0$ where $\Omega$ is the augmentation ideal of $\Z H'$. Then $A$ has a filtration of $\Z H$-modules
$$
0 = A_s = A \Omega^s \subseteq \ldots \subseteq A_{i+1} = A \Omega^{i+1} \subseteq A_i = A \Omega^i \subseteq \ldots \subseteq A_0 = A
$$
where $H'$ acts trivially on
$A_i / A_{i+1}$ via conjugation i.e. $A_i / A_{i+1}$ is a $\Z Q$-module, where $Q = H / H'$.
Since $A$ is a finitely generated module over a Noetherian ring $\Z H$, we have that every $A_i / A_{i+1}$ is a finitely generated as a $\Z Q$-module.

Suppose that $A_{i+1}$ is finitely generated as an additive group. We will prove that $A_i$ is finitely generated as an additive group. We follow the method used in the proof of Theorem \ref{class2}.
Let $P_1, \ldots, P_s$ be the minimal associated primes for the $\Z Q$-module $A_i / A_{i+1}$, thus $M = \Z Q/ P_1 \oplus \ldots \oplus \Z Q / P_s $ is a submodule of  $A_i / A_{i+1}$.

Let $\tilde{a}$ be the element of $M$ whose projection to $\Z Q / P_j$ is the unity element for every $1 \leq j \leq s$ and $a$ be a preimage of $\tilde{a}$ in $A$.
  Note that $A \subseteq C_G(a)$  and define $H_0 = C_G(a) / A \subseteq H$.
  Since $C_G(a)$ is finitely generated $A$ is finitely generated as a $\Z H_0$-module, so $A_i / A_{i+1}$ is finitely generated as a $\Z Q_0$-module and so $M$ is finitely generated as a $\Z Q_0$-module, where $Q_0 = H_0 / (H_0 \cap H')$. By the choice of $a$ we get that $Q_0$ acts trivially on $M$, thus $M$ is finitely generated as an additive group and by Lemma \ref{commutativealgebra} $A_i / A_{i+1}$ is finitely generated as an additive group.

\end{proof}

\section{Conjugacy classes}

\noindent
Let $T$ and $G$ be groups with $T$ acting on $G$. We consider the
following condition:

\begin{itemize}
\item[(1T)] $G$ has the maximal condition for virtually cyclic subgroups (max-vc)
and only finitely many $T$-orbits of maximal virtually cyclic
subgroups.
\end{itemize}

\noindent
If $G=T$ acts on itself by conjugation, then (1T) is equivalent to condition (1) of Section 2.
Note also that max-vc is subgroup closed and that for any $x\in T$ and any $H\leq G$ maximal virtually
cyclic, the group $H^x$ is also maximal virtually cyclic.

\begin{lemma}\label{infprim} Let $h(x)\in\Z[x]$ be an integer
polynomial. There are infinitely many primes $p$ such that there
exists an integer $n_p$ with
$$p|h(n_p).$$
\end{lemma}
\begin{dem} Consider prime factors of $h(n)$ when $n\in\Z$. If we had only a
finite set $\{p_1,\ldots,p_r\}$, then for each $i$ let $n_i$ such
that
$$p_i^{n_i-1}|h(0),\quad p_i^{n_i}\not|h(0).$$

Note that for any $m$ with $\prod_{i=1}^rp_i^{n_i}|m$,
$$p_i^{n_i}\not|h(m).$$
So we may choose $m$ big enough so that $h(m)$ has some other prima
factor $q$ not in the finite set above.\end{dem}

\begin{lemma}\label{abelian} Let $A$ be an abelian group and  $T$ be an infinite cyclic group acting on $A$.
Suppose $A$ has property (1T).  Then $A$ is finite.
\end{lemma}
\begin{proof} Note first that the condition (1T) implies that the
order of the torsion elements of $A$ is bounded. Therefore there are
only finitely many primes $p$ for which the $p$-primary component of
$A$ is non-trivial, so by splitting $A$ in its torsion free part and
its $p$-primary components we may assume $A$ is either torsion free
or that all its elements have $p$-power order for a fixed prime $p$.

 If we denote $a_1,\ldots, a_s$ for representatives
of the $T$-orbits of maximal cyclic subgroups in $A$, then there is
an epimorphism
$$\Z T\oplus\buildrel s\over\ldots\oplus\Z T\twoheadrightarrow A$$
and each $\Z Ta_i\leq A$ has also (1T), so we may assume $$A=\Z
Ta\cong\Z T/I$$ for certain ideal $I\normal\Z T.$ Let $m\Z=I\cap\Z$
(the assumption above implies that $m$ is either 0 or a power of
$p$). As $A=\Z T/I$ has only finitely many $T$-orbits of maximal
virtually cyclic subgroups, $I\neq m\Z T$. Therefore we may choose
some polynomial
$$f(t)\in I\cap\Z[t]\setminus m\Z[t]$$
of minimal degree $k$ where $t$ is a fixed generator of $T$. Note
that $k>0$. Then the family $\{1,t,t^2,\ldots,t^{k-1}\}$ yields a
maximal independent subset of the quotient $A=\Z T/I$. Therefore $A$
has finite Pr\"ufer rank so max-vc implies that it is a finitely
generated abelian group. In the torsion case we get the result so
from now on we assume $A$ is torsion free.

There is a principal ideal $I_1\normal\Z T$ such that $I\leq I_1$
and $I_1/I$ is torsion.
Indeed $I \otimes_{\Z} {\Q}$ is an ideal of the principal ideal domain $\Q T$, hence is generated by some element $h(t)$. In fact $h(t)$ can be chosen to be a primitive polynomial in the variable $t$. A localisation version of Gauss' lemma implies $\Z T \cap (\Q T h(t)) = \Z T h(t)$, and we can define $I_1 = \Z T h(t)$.
So our assumption that $A$ is torsion free implies that
$I=\Z T h(t)$.

Now, by Lemma \ref{infprim} there is an  infinite set of primes
$\Omega$ such that for each $p\in\Omega$ there exists an integer
$n_p$ with
$$t-n_p|h(t)\text{ mod }p.$$ Any of the
elements $t-n_p+I$ generates a maximal cyclic subgroup of $A$. So
our assumption on $A$ implies that there are only finitely many
$T$-orbits between them. And therefore there are only finitely many
distinct ideals in the set
$$\{I_p=\Z T(t-n_p,h(t)):p\in\Omega\}.$$
Moreover,
$$\Z T/I_p\cong\Z/h(n_p)\Z$$
is finite so there are only finitely many maximal ideals containing
some ideal in the set above.

But for each prime $p\in\Omega$,
$$\Z T(t-n_p,h(t))\leq M_p=\Z T(t-n_p,h(t),p)$$
which is a maximal ideal and obviously $M_p\neq M_q$ for $p\neq q.$
This is a contradiction.
\end{proof}

\noindent From now on let $G=T$ act on itself by conjugation and we shall consider the conditions (1) and (2) of Section 2.

\begin{lemma}\label{finindex} Let $H\normal G$ be a finite index normal subgroup.
If $G$ has property (1), then so does $H$.\end{lemma}
\begin{proof} Clearly, $H$ has max-vc. Moreover for any $C\leq H$
maximal virtually cyclic there is some $D\leq G$ also maximal
virtually cyclic with $C\leq D$ and from this follows that $D\cap
H=C.$ As the index $|G:N_G(D)H|$ is finite, the $G$-conjugacy class
of $D$ splits after intersecting with $H$ as finitely many
$H$-conjugacy classes of maximal virtually cyclic subgroups of $H$.
\end{proof}

\begin{lemma}
\label{quotients1} Assume that $G$ has (1) and there is some
$N\normal G$ with $G/N$ finitely generated abelian. Then $G/N$ is
virtually cyclic.
\end{lemma}
\begin{proof} By Lemma \ref{finindex} we may assume that $G/N$ is torsion
free. Note that for any $\bar C=<\bar x>\leq G/N$ maximal cyclic we
may choose an $x\in G$ with $xN=\bar x$ and a maximal virtually
cyclic $C_x$ of $G$ with $x\in C_x.$ And $C_x^g=C_y$ implies
$x^gN=yN.$ Therefore there are finitely many conjugacy classes of
maximal cyclic subgroups of $G/N$. As this is torsion free abelian
we deduce that $G/N$ is cyclic.
\end{proof}

\begin{lemma}\label{quotients2} Assume that $G$ has property (1) and there is $N\normal G$  finite.
Then $G/N$ has property (1).
\end{lemma}
\begin{proof} For any $N\leq C\leq G$ such that $C/N$ is virtually cyclic, $C$ itself is virtually
cyclic. This yields the result.
\end{proof}

\begin{thm} \label{quotientnaf} Let $G$ be finitely generated
 nilpotent-by-abelian-by-finite with property (1). Then $G$ is virtually cyclic.
\end{thm}
\begin{proof} By Lemma \ref{finindex} we may assume $G$ is nilpotent-by-abelian.
 Let $N\normal G$ be nilpotent with $G/N$ abelian. As $T=G/N$ is finitely generated by
 Lemma \ref{quotients1} it must be
 cyclic. Let $A=Z(N).$ Then $G/A$ acts on $A$ and $N/A$ is in the
 kernel of the action. Moreover the condition (1) implies that $A$
 has property (1T) with respect to the group $T = G / N$. So by Lemma \ref{abelian} $A$ must
 be finite. And Lemma \ref{quotients2} implies that $G/A$ also has (1) so an induction on the
 nilpotency length of $N$ yields the result.
 \end{proof}

\begin{cor} \label{conjclassesAP} Let $G$ be an abelian-by-polycyclic-by-finite group of type $\FP_3$ and with property (1). Then $G$ is virtually cyclic.
\end{cor}

\begin{proof} By \cite{GKR} abelian-by-polycyclic groups of type $\FP_3$ are nilpotent-by-abelian-by-finite and we can apply the previous Theorem.
\end{proof}

\begin{cor} \label{conjclassesPS} Let $G$ be polycyclic-by-finite group with property (2). Then $G$ is virtually cyclic.
\end{cor}
\begin{proof} As $G$ is polycyclic-by-finite, it has max and hence max-vc. Therefore by Lemma
\ref{max} $G$ has (1). Moreover, polycyclic groups are of type $\FP_\infty$ (\cite[2.6]{Bieribook}) so as in Corollary \ref{centraliserFP} our group $G$ is virtually nilpotent-by-abelian. Now it suffices to apply Theorem \ref{quotientnaf}.
\end{proof}

\section{The main results}


\begin{thm} \label{mainresult}
Soluble groups of type $\UUFP_\infty$ such that centralizers of infinite order elements are of type $\FP_\infty$ are virtually cyclic.
\end{thm}

\begin{proof} Proposition \ref{uufpn} and Corollary \ref{centraliserFP} imply that the group is polycyclic and satisfies property (2). Now apply  Corollary \ref{conjclassesPS}.
\end{proof}

\begin{thm} Let $G$ be an abelian-by-nilpotent group of type $\UUFP_{~1}$ such that centralizers of infinite order elements are finitely generated. Then $G$ is virtually cyclic.
\end{thm}

\begin{proof} By  Proposition \ref{uufpn} centralizers of arbitrary elements are finitely generated. \cite[Thm.~C]{Lennox2} implies $G$ is polycyclic. Now apply Corollary \ref{conjclassesPS}.
\end{proof}

\begin{thm} Let $G$ be a (nilpotent of class 2)-by-abelian group of type $\UUFP_{~2}$ such that centralizers of infinite order elements are of type $\FP_2$. Then $G$ is virtually cyclic.
\end{thm}

\begin{proof} This follows from Proposition \ref{uufpn}, Corollary \ref{class2} and Theorem \ref{quotientnaf}.
\end{proof}

\begin{thm} Let $G$ be an abelian-by-polycyclic group of types $\FP_3$ and $\UUFP_{~1}$ such that centralizers of infinite order elements are finitely generated. Then $G$ is virtually cyclic.
\end{thm}

\begin{proof} This follows from Proposition \ref{uufpn}, Theorem \ref{FP3} and Corollary \ref{conjclassesAP}.
\end{proof}

\begin{cor}\label{last} Let $G$ be an abelian-by-polycyclic group of type $\UUFP_{~3}$  such that centralizers of infinite order elements are finitely generated. Then $G$ is virtually cyclic.
\end{cor}

  \begin{proof} By Proposition \ref{uufpn} every group of type $\UUFP_{~3}$ is of type $\FP_3$.
  \end{proof}

  \medskip\noindent We shall conclude with a remark on virtual notions. Let $\mathcal C$ be a class of groups. We say a group is virtually $\mathcal C$ if $G$ has a finite index subgroup belonging to $\mathcal C.$

 \begin{lemma}\label{finind}
 Let $G_1$ be a group of type $\UUFP_n$, $n \geq 0$. Then any subgroup $G$ of finite index in $G_1$
 is of type $\UUFP_n.$
 \end{lemma}

 \begin{proof}
 Any $\OVCG_1$-projective resolution of $\uuz$ can be viewed as an $\OVCG$-projective resolution.
 We therefore need to show that every finitely generated $\OVCG_1$-projective is finitely generated as
 $\OVCG$-module. This reduces to showing that $\Z[\blah, G_1/K]$ is a finitely generated $\OVCG$-projective for all virtually cyclic subgroups $K$ of $G_1$. $G_1/K$ viewed as a $G$-set is the $G$-finite
 set $\Omega = \sqcup G/G\cap K^{x^{-1}}$, where $x$ runs through a set of representatives of the finite number of cosets of $G$ in $G_1.$
\end{proof}

\begin{remark} In the statements of Theorem \ref{mainresult} through to Corollary \ref{last} above we can
let $G$ be virtually-$\mathcal C$, where $\mathcal C$ stands for soluble, abelian-by-nilpotent, (nilpotent of class 2)-by-abelian or abelian-by-polycyclic respectively.
\end{remark}

\noindent As mentioned in the introduction, we can also strengthen Theorem \ref{mainresult}
to apply to the class of elementary amenable groups:

\begin{cor} Elementary amenable groups of type $\UUFP_\infty$  such that centralizers of infinite order elements are $\FP_\infty$ are virtually cyclic.
\end{cor}

\begin{proof}
These groups are of type $\FP_\infty$ by Proposition \ref{uufpn} and hence have finite Hirsch length and a bound on the orders of their finite subgroups \cite{kropholler93b}. Combining this with  \cite{hillmanlinnell}, we might  assume that the group is virtually soluble. Now apply Lemma \ref{finind}.
\end{proof}


\end{document}